\documentclass[11pt]{amsart}
\usepackage{amscd}
\usepackage{graphicx}
\def\rt{\rightarrow}

\def\a{\alpha}
\def\l{\lambda}

\def\pl{\partial}
\def\ml{\mathcal{H}}
\def\mb{\mathbb{R}}
\def\mbn{\mathbb{R}^{2n+1}}
\def\b{B(z_0,2\rho)}
\def\s{\subseteq}
\def\zh{\widehat{z}}
\def\uh{\widehat{u}}
\def\xh{\widehat{x}}
\def\hp{\widehat{p}}

\def\opn#1#2{\def#1{\operatorname{#2}}} 
\opn\chara{char} \opn\length{\ell} \opn\pd{pd} \opn\rk{rk}
\opn\projdim{proj\,dim} \opn\injdim{inj\,dim} \opn\rank{rank}
\opn\depth{depth} \opn\grade{grade} \opn\height{height}
\opn\embdim{emb\,dim} \opn\codim{codim}

\opn\Tr{Tr} \opn\bigrank{big\,rank}
\opn\superheight{superheight}\opn\lcm{lcm}
\opn\trdeg{tr\,deg}
\opn\reg{reg} \opn\lreg{lreg} \opn\ini{in} \opn\lpd{lpd}
\opn\size{size}

\opn\Spec{Spec} \opn\Supp{Supp} \opn\supp{supp} \opn\Sing{Sing}
\opn\Ass{Ass} \opn\Min{Min}
\opn\Ann{Ann} \opn\Rad{Rad} \opn\Soc{Soc}
\opn\Im{Im} \opn\Ker{Ker} \opn\Coker{Coker} \opn\Am{Am}
\opn\Hom{Hom} \opn\Tor{Tor} \opn\Ext{Ext} \opn\End{End}
\opn\Aut{Aut} \opn\id{id}

\opn\nat{nat}
\opn\pff{pf}
\opn\Pf{Pf} \opn\GL{GL} \opn\SL{SL} \opn\mod{mod} \opn\ord{ord}
\opn\Gin{Gin} \opn\Hilb{Hilb}\opn\sdepth{sdepth}
\opn\aff{aff} \opn\con{conv} \opn\relint{relint} \opn\st{st}
\opn\lk{lk} \opn\cn{cn} \opn\core{core} \opn\vol{vol}
\opn\link{link} \opn\star{star}
\opn\gr{gr}

\def\pot#1#2{#1[\kern-0.28ex[#2]\kern-0.28ex]}

\opn\dirlim{\underrightarrow{\lim}}
\opn\inivlim{\underleftarrow{\lim}}
\newtheorem{theorem}{THEOREM}[section]

\newtheorem{lemma}[theorem]{LEMMA}
\newtheorem{remark}[theorem]{\textit{Remark}}
\newtheorem{definition}[theorem]{\textit{Definition}}

\let\epsilon\varepsilon
\let\phi=\varphi
\let\kappa=\varkappa
\def\qed{\ifhmode\textqed\fi
      \ifmmode\ifinner\quad\qedsymbol\else\dispqed\fi\fi}
\def\textqed{\unskip\nobreak\penalty50
       \hskip2em\hbox{}\nobreak\hfil\qedsymbol
       \parfillskip=0pt \finalhyphendemerits=0}
\def\dispqed{\rlap{\qquad\qedsymbol}}

\opn\dis{dis}
\def\pnt{{\raise0.5mm\hbox{\large\bf.}}}

\opn\Lex{Lex}

\usepackage{sidecap}

\begin{document}
\textwidth=13cm

\title{Parameterized Stationary Solution for first order PDE }

\author {SAIMA PARVEEN$^{1}$ and MUHAMMAD SAEED AKRAM $^{2}$ }
 \maketitle
 \pagestyle{myheadings} \markboth{\centerline {\scriptsize
SAIMA PARVEEN, MUHAMMAD SAEED AKRAM   }} {\centerline {\scriptsize
 Parameterized Stationary Solution for PDE of first order}}

We analyze the existence of a parameterized stationary solution $z(\lambda,z_0)=\\\big(x(\lambda,z_0), p(\lambda,z_0),\,u(\lambda,z_0)\big)\in D\subseteq\mathbb{R}^{2n+1},\,\lambda\in B(0,a)\subseteq\mathop{\prod}\limits_{i=1}^{m}[-a_i,a_i]$, associated with a nonlinear first order PDE, $H_0(x,p(x),u(x))=\hbox{constant}\,\,(p(x)=\partial_x u(x))$ relying on\\
(a) first integral $H\in\mathcal{C}^\infty\big( B(z_0,2\rho)\subseteq\mathbb{R}^{2n+1}\big)$ and the corresponding Lie algebra of characteristic fields is of the finite type;\\
(b) gradient system in a Lie algebra finitely generated over orbits
$(f.g.o;z_0)$ starting from $z_0\in D$ and their nonsingular
algebraic representation.

\maketitle
\section{INTRODUCTION}
Let $H_0(x,p,u),\, z=(x,p,u)\in B(z_0.2\rho)\subseteq\mbn$, be a
second order continuously differentiable function,
$H_0\in\mathcal{C}^2\big( B(z_0,2\rho)\subseteq\mbn\big)$  and
consider the equation
\begin{equation}\label{1.1}
H_0(z)=H_0(z_0)\,\,\hbox{for}\,\,z=(x,p,u)\in D\s\mbn,\,z_0\in D\,.
\end{equation}
In other words find a manifold $D\s\b\s\mbn$ such that the equation
\eqref{1.1} is satisfied for any $z\in D$. In the case that manifold
$D\s\mbn$ can be described as follows
\begin{equation}\label{1.2}
D=\big\{\big(x,p(x),u(x)\big)\in\mbn :\, p(x)=\pl_x u(x), x\in
B(x_0,\rho\s\mathbb{R}^n)\big\} ,
\end{equation}
we call it as the standard stationary solution associated with the
nonlinear first order PDE given in \eqref{1.1}.

It relies on the flow $\{\widehat{z}(t,\l)\in\mbn \,:\,
\widehat{z}(0,\l)=\widehat{z}_0(\l),\,
t\in(-a,a),\,\l\in\Lambda\s\mb^{n-1}\}$ generated by the
characteristic field $Z_0(z)\mathop{=}\limits^{\hbox{def}}\big(\pl_p
H_0(z), P_0(z), <p, \pl_p H_0(z)>\big), \, P_0(z)=-[\pl_x
H_0(z)+p\pl_u H_0(z)]$, of the smooth scalar function $H_0\in
\mathcal{C}^2(\mbn)$. Using that $H_0$ is a first integral for the
characteristic field $Z_0$, we get
\begin{equation}\label{1.3}
H_0\big(\widehat{z}(t,\l)\big)=
H_0\big(\widehat{z}_0(\l)\big)\,,\,\,t\in(-a,a),\,\l\in\Lambda\s\mathbb{R}^{n-1}\,,
\end{equation}
for any parameterized Cauchy conditions
\begin{equation}\label{1.4}
\widehat{z}_0(\l)=\big(\widehat{x}_0(\l), \widehat{p}_0(\l),
\widehat{u}_0(\l)\big)\,,
\end{equation}
fulfilling the compatibility condition
\begin{equation}\label{1.5}
\pl_{\l_i}\widehat{u}_0(\l)=\big<
p_0(\l),\pl_{\l_i}\widehat{x}_0(\l)\big>\,,\,\,i=1,\ldots,n-1,\,\,\l=(\l_1,\ldots,\l_{n-1})\in\Lambda\s\mathbb{R}^{n-1}\,.
\end{equation}
The standard stationary solution for \eqref{1.1} can be obtained
imposing the following new constraints
\begin{equation}\label{1.6}
H_0\big(\widehat{z}_0(\l)\big)=H_0(z_0)\,,\,\,\,\l\in\Lambda\s\mathbb{R}^{n-1}\,,
\end{equation}
and
\begin{equation}\label{1.7}
\hbox{the $n$ vectors in} \,\mathbb{R}^n,\{\pl_p
H_0\big(\widehat{z}_0(\l)\big),
\pl_{\l_1}\widehat{x}_0(\l),\ldots,\pl_{\l_{n-1}}\widehat{x}_0(\l)\}\s\mathbb{R}^n
\end{equation}
are linearly independent for any $\l\in\Lambda\s\mathbb{R}^{n-1}$.
One may notice that the last conditions need to take into
consideration very special $H_0\in\mathcal{C}^2(\mbn)$ and Cauchy
conditions $\{\widehat{z}_0(\l):\, \l\in\Lambda\s\mathbb{R}^{n-1}\}$
such that \eqref{1.7} is fulfilled. Here we propose to construct a
parameterized version of Cauchy conditions such that
\begin{equation}\label{1.8}
\zh(\l;z_0)=\big(\xh(\l;z_0),\hp(\l;z_0),\uh(\l;z_0)\big)
\end{equation}
satisfies stationary conditions
\begin{equation}\label{1.9}
H_0\big(\zh(\l;z_0)\big)=H_0(z_0)\,,\,\,\l\in\Lambda=\mathop{\prod}\limits_{i=1}^{m}[-a_i,a_i]\,,\,\,\widehat{z}(0,z_0)=z_0\,.
\end{equation}
In addition, the solution in \eqref{1.8} satisfying stationarity
conditions \eqref{1.9}, will be obtained as a finite composition of
flows starting from $z_0\in\mbn$ (orbit of the origin $z_0\in\mbn$)
generated by some characteristic fields including $Z_0$.

In the case that the nonsingularity conditions \eqref{1.7} are
fulfilled $(m=n)$ then the parameterized version lead us to a
standard stationary solution. A first order continuously
differentiable $\zh(\l;z_0): \Lambda\s\mathbb{R}^m\rightarrow\mbn$
satisfying stationarity conditions \eqref{1.9} will be called a
parameterized stationary solution for PDE \eqref{1.1}.

For solving nonlinear equation \eqref{1.9}, we use the following
procedure. First, the nonlinear equations \eqref{1.9} is transformed
into a first order linear system of PDE where the unknowns are the
characteristic vector fields generated by a finite set of first
integrals $\{H_1(z),\ldots,H_m(z): z\in B(z_0,2\rho)\s\mbn\}$
corresponding to $Z_0$. Then look for a solution of \eqref{1.9} as
an orbit in the Lie algebra of characteristic fields generated by
$\{Z_1(z),\ldots,Z_m(z):\, z\in B(z_0,2\rho)\s\mbn\}$ corresponding
to the first integrals $\{H_1(z),\ldots, H_m(z):\,z\in
B(z_0,2\rho)\s\mbn\}$.

In the particular case when PDE \eqref{1.1} is determined by a
function $H_0(x,p)$, the construction of a parameterized stationary
solution is analyzed in Theorem \ref{th.1} of section 3. For the
general case, the result is given in Theorem \ref{th.2} of section
3.

In the section 2 are included all definitions and some auxiliary
results necessary for the main results given in section 3.

The method of using finite composition of flows (orbit) and the
corresponding gradient system in a Lie algebra of vector fields has
much in common with the references included here (see \cite{01},
\cite{02} and \cite{03}) where both parabolic equations with
stochastic perturbations and overdetermined system of first order
PDE are studied.

This paper is intended to be a new application of the
geometric-algebraic methods presented in \cite{03}.
 \section{Definitions, formulation of problems and some auxiliary results}
 Denote $\mathcal{H}=\mathcal{C}^\infty(\mbn,\mathbb{R})$ the space consisting of the scalar functions $H(x,p,u):\mathbb{R}^n\times\mathbb{R}^n\times\mathbb{R}\rightarrow\mathbb{R}$ which are continuously differentiable of any order. For each pair  $H_1,H_2\in\mathcal{H}$, define the Poisson bracket
 \begin{equation}\label{1.10}
 \{H_1,H_2\}(z)=\big<\pl_z H_2(z),Z_1(z)\big>,\,\,z=(x,p,u)\in\mbn
\end{equation}
where $\pl_z H_2(z)$ stands for the gradient of a scalar function
$H_2\in\mathcal{H}$ and $Z_1(z)=\big(X_1(z),
P_1(z),U_1(z)\big)\in\mbn,\,z\in\mbn$ is the characteristic field
corresponding to $H_1\in\mathcal{H}$. We recall that $Z_1$ is
obtained from $H_1\in\mathcal{H}$ such that the following equation
\begin{equation}\label{1.11}
X(z)=\pl_p H_1(z),\,P_1(z)=-[\pl_x H_1(z)+p\pl_u
H_1(z)],\,U_1(z)=\big<p, H_1(z)\big>
\end{equation}
are satisfied. The linear mapping connecting an arbitrary
$H\in\mathcal{H}$ and its characteristic field can be represented by
\begin{equation}\label{1.12}
Z_H(z)=T(p)\big(\pl_z H\big)(z)\,,\,\,z\in\mbn\,.
\end{equation}
Here the real $(2n+1)\times(2n+1)$ matrix $T(p)$ is defined by
\begin{equation}\label{1.13}
T(p)=\left(
       \begin{array}{ccc}
         O & I_n & \theta \\
         -I_n & O & -p \\
         \theta^* & p^* & O \\
       \end{array}
     \right)
\end{equation}
considering $O=\hbox{zero marrix of}\,M_{n,m},\,I_n- $ unity matrix
of $M_{n,m}$ and $\theta\in\mathbb{R}^n$ is the null column vector.

We notice that $T(p)$ is a skew symmetric
\begin{equation}\label{1.14}
[T(p)]^*=-T(p)
\end{equation}
and as a consequence, the Poisson bracket satisfies a skew symmetric
property
\begin{equation}\label{1.15}
  \begin{array}{ll}
    \{H_1,H_2\}(z)&=\big<\pl_z H_2(z),Z_1(z)\big>=\big<\pl_z H_2(z),T(p)(\pl_z H_1)(z)\big>\\
   &=\big<[T(p)]^*(\pl_z H_2)(z),\pl_zH_1(z)\big>=-\{H_2,H_1\}(z)
  \end{array}
\end{equation}

In addition, the linear space of characteristic fields
$K\s\mathcal{C}^\infty(\mbn;\mbn)$ is the image of a linear mapping
$S:D\mathcal{H}\rightarrow K$ , where $D\mathcal{H}=\{\pl_z
H=H\in\mathcal{H}\}$. Using \eqref{1.12}, we define
\begin{equation}\label{1.16}
S(\pl_z H)(z)=T(p)(\pl_z H)(z)\,,\,\, z\in\mbn
\end{equation}
where the matrix $T(p)$ is given in \eqref{1.13}.

The linear space of characteristic fields
$K=S(D\mathcal{H\mathcal{}})$ is extended to a Lie algebra
$L_k\s\mathcal{C}^\infty(\mbn;\mbn)$, using the standard Lie bracket
of vector fields
\begin{equation}\label{1.17}
[Z_1,Z_2]=[\pl_z Z_2]Z_1-[\pl_z Z_1]Z_2\,,Z_i\in K,\,i=1,2\,.
\end{equation}
On the other hand, each $H\in\mathcal{H}$ is associated with a
linear mapping
\begin{equation}\label{1.18}
\overrightarrow{H}(\phi)(z)=\{H,\phi\}(z)=\big<\pl_z\phi(z),Z_H(z)\big>\,,\,\,z\in\mbn\,,
\end{equation}
for each $\phi\in\mathcal{H}\mathcal{}$, where $Z_H\in K$ is the
characteristic vector field corresponding to $H\in\mathcal{H}$
obtained from $\pl_z H$ by $Z_H(z)=T(p)(\pl_z H)(z)$ see
\eqref{1.12}.

Define a linear space consisting of linear mappings
\begin{equation}\label{1.19}
\overrightarrow{\ml}=\{\overrightarrow{H}\,:\,H\in \ml\}
\end{equation}
 and extend $\overrightarrow{\mathcal{H}}$ to a Lie algebra $L_H$ using the Lie bracket of linear mappings
\begin{equation}\label{1.20}
[\overrightarrow{H}_1,\overrightarrow{H}_2]=\overrightarrow{H}_1\circ\overrightarrow{H}_2-\overrightarrow{H}_1\circ\overrightarrow{H}_1\,.
\end{equation}
The link between the two lie algebras $L_K$ (extending $K$) and
$L_H$ (extending $\overrightarrow{\mathcal{H}}$) is given by a
homomorphism of Lie algebras
\begin{equation}\label{1.21}
A: L_H\rightarrow L_K\,,\,\, A(\overrightarrow{\mathcal{H}})=K
\end{equation}
satisfying
\begin{equation}\label{1.22}
A\big([\overrightarrow{H}_1,\overrightarrow{H}_2]\big)=[Z_1,Z_2]\in
L_K\,,\,\,\hbox{where}\,\,Z_i=A(\overrightarrow{H}_i),\,i=1,2\,.
\end{equation}
\begin{remark}\label{r.1}
In general, the lie algebra
$L_H\supseteq\overrightarrow{\mathcal{H}}$ does not coincide with
the linear space $\overrightarrow{\mathcal{H}}$ and as a
consequence, we get $K\s L_K,\,\,L_K\neq K$. It relies upon the fact
that the linear mapping $\{\overrightarrow{H_1,H_2}\}$ generated by
the Poisson bracket $\{H_1,H_2\}\in\mathcal{H}$ does not coincide
with the Lie bracket $[\overrightarrow{H}_1,\overrightarrow{H}_2]$
defined in \eqref{1.20}.
\end{remark}
\begin{remark}\label{r.2}
In the particular case when $H_0$ in \eqref{1.1} is replaced by a
second order continuously differentiable function
$H_0(x,p):\mathbb{R}^{2n}\rightarrow\mathbb{R}$ then the above given
analysis will be restricted to the space
$\mathcal{H}=\mathcal{C}^\infty(\mathbb{R}^{2n},\mathbb{R})$. If it
is the case then the corresponding linear mapping $S:D\mathcal{H}\rt
K$ is determined by a simplectic matrix
\begin{equation}\label{1.23}
\widehat{T}=\left(
              \begin{array}{cc}
                O & I_n \\
                -I_n & O \\
              \end{array}
            \right)\,,\,\,D\ml=\{\pl_z H:\, H\in\mathcal{C}^\infty(\mathbb{R}^{2n};\mathbb{R})\}\,.
\end{equation}
In addition, the linear spaces $\overrightarrow{\ml}$ and
$K\s\mathcal{C}^\infty(\mathbb{R}^{2n};\mathbb{R}^{2n})$ coincide
with their Lie algebras $L_H$ and correspondingly $L_K$. It follows
from a direct computation and we get
\begin{equation}\label{1.24}
[Z_1,Z_2](z)=\widehat{T}\pl_z H_{12}\,,\,\,z\in \mb^n
\end{equation}
where $Z_i=T\pl_z H_i\,,i=1,2$ and
\begin{equation}\label{1.25}
H_{12}=(z)=\{H_1,H_2\}(z)=\big<\pl_z H_2(z), Z_1(z)\big>
\end{equation}
in the Poisson bracket associated with two scalar functions
$H_1,H_2\in\ml$. Compute
\begin{equation}\label{1.26}
\begin{array}{ll}
    \widehat{T}H_{12}(z)&= \widehat{T}[\pl_{z}^2 H_2(z)]Z_1(z)+\widehat{T}[\pl_{z} Z_1^*(z)]\pl_z H_2(z)\\
    &=[\pl_{z}Z_2(z)]Z_1(z) +\widehat{T}[(\pl_{z} H_1)^*(z)\widehat{T}^*]\pl_zH_2(z)\\
    &=\pl_{z} Z_2(z)Z_1(z)-\widehat{T}(\pl_{z}^2 H_1)(\widehat{T}\pl_zH_2(z)) \\
    &=[\pl_{z} Z_2(z)]Z_1(z)-[\pl_{z} Z_1(z)]Z_2(z)=[Z_1,Z_2](z)
  \end{array}
\end{equation}
and the conclusion $\{L_H=\overrightarrow{\ml}\,,\,\,K_k=K\}$ is
proved. For the general case, the conclusion of Remark \ref{r.2} is
not any more true and the manifold structure involved in the
solution of PDE \eqref{1.1} will be obtained, using first integrals
$\{H_1(z),\ldots,H_m(z):\, z\in\b\} \,(m\geq1)$ corresponding to the
fixed characteristic field $Z_0$.
\end{remark}
Denote by $K_0\s\mathcal{C}^\infty(\b;\mbn)$ the linear space
consisting of all characteristic fields $Z\in
\mathcal{C}^\infty(\b;\mbn)$
\begin{equation}\label{1.27}
Z(z)=T(p)\pl_z H(z)\,,\,\,\hbox{where}\,\,\{H(z):\,z\in\b\s\mbn\}
\end{equation}
is a first integral for the fixed characteristic vector field
$\{Z_0(z):\,z\in\b\}$. Let $L_0\s\mathcal{C}^\infty(\b;\mbn)$ be the
Lie algebra determined by the linear space $K_0\s L_0$.
\begin{definition}\label{def.1}
We say that $L_0$ is of the finite type over
$\mathcal{C}^\infty(\b)(\hbox{or}\,\mb)$ with respect to $K_0$ if
there exists a system of vector fields
$\{Z_1(z),\ldots,Z_m(z):\,z\in\b\}\s K_0$ such that any Lie bracket
$[Z_i,Z_j](z)=\mathop{\sum}\limits_{k=1}^{m}\a_{ij}^k(z)Z_k(z),\,\,z\in\b$,
where
$\a_{ij}^{k}\in\mathcal{C}^\infty(\b)(\,\hbox{or}\,\,\a_{ij}^{k}\in\mb),\,i,j,k\in\{1,\ldots,m\};\\
\,\{Z_1,\ldots,Z_m\}\s\mathcal{C}^\infty(\b;\mbn)$
will be called a system of generators for $L_0$.
\end{definition}
\begin{remark}\label{r.3}
In the case that PDE \eqref{1.1} is determined by a scalar function
$H_0(x,p):\b\s\mb^{2n}\rt\mb$ then the linear space
$\widehat{K}_0\s\mathcal{C}(\b\s\mb^{2n};\mb^{2n})$ is consisting of
all characteristic fields
\begin{equation}\label{1.28}
\widehat{Z}(z)=\widehat{T}\pl_z H(z)
\end{equation}
where $H(z),\,z\in\b$, is first integral of
$\widehat{Z}_0(z)=\widehat{T}\pl_z H_0(z)$ (see Remark \ref{r.2}).
We get that the Lie algebra $\widehat{L}_0$ determined by
$\widehat{K}_0$ coincides with
$\widehat{K}_0,\,\widehat{L}_0=\widehat{K}_0$ (see Remark
\ref{r.2}).
\end{remark}
\begin{definition}\label{d.2}
We say that $\widehat{L}_0$ is of the finite type over
$\mathcal{C}^\infty(\b)$ (or $\mb$) if there exists a system of
vector fields $\{Z_1(z),\ldots,Z_m(z):\,z\in\b\s\mb^{2n}\}\s
\widehat{K}_0$ such that any Lie bracket
$[Z_i,Z_j](z)=\mathop{\sum}\limits_{k=1}^{m}\a_{ij}^k(z)Z_k(z)$, for
$\a_{ij}^{k}\in\mathcal{C}^\infty(\b)(\,\hbox{or}\,\,\a_{ij}^{k}\in\mb),\,i,j,k\in\{1,\ldots,m\};\,\{Z_1,\ldots,Z_m\}$
is called a system of generators for $\widehat{L}_0$.
\end{definition}
\begin{remark}\label{r.4}
The simplest case in our analysis is obtained when PDE \eqref{1.1}
is determined by a linear function with respect to $p\in\mb^n$, i.e
\begin{equation}\label{1.29}
H_0(x,p)=\big<p,f_0(x)\big>\,,\,\,p\in\mb^n,\,x\in\mb^n,\,z=(x,p),\,z_0=(x_0,p_0)
\end{equation}
where $f_0\in\mathcal{C}^2(\mb^n,\mb^n)$. In this case the linear
space $\widehat{K}_0$ (see Remark \ref{r.3}) coincides with the Lie
algebra $\widehat{L}_0$ generated by $\widehat{K}_0$ where
\begin{equation}\label{1.30}
\widehat{K}_0=\{\widehat{T}\pl_z
H(z):\,H(z)=\big<p,f(x)\big>,\,[f,f_0](x)=0\,,x\in\b\s\mb^n\}.
\end{equation}
\end{remark}
\begin{lemma}\label{l.1}
Assume that PDE \eqref{1.1} is determined by
$H_0(x,p)=\big<p,f(x)\big>,\,\,x\in\b\s\mb^n,\,f_0\in\mathcal{C}^2(\b;\mb^n)$.
Assume that there exists
$\{f_1,\ldots,f_m\}\s\mathcal{C}^\infty(\b;\mb^n)$ satisfying
\begin{equation}\label{1.31}
[f_i,f_j](x)=\mathop{\sum}\limits_{k=1}^m\a_{ij}^k
f_k(x),\,x\in\b,\,\,\hbox{where}\,\,\a_{ij}^k\s\mb,\,\,k,i,j\in\{1,\ldots,m\}
\end{equation}
\begin{equation}\label{1.32}
[f_0,f_i](x)=0,\,\,x\in\b,\,\,i\in\{1,\ldots,m\}\,.
\end{equation}
Then the lie algebra $\widehat{L}_0$ generated by $\widehat{K}_0$
(see \eqref{1.30}) fulfils
\begin{equation}\label{1.33}
\widehat{L}_0=\widehat{K}_0\,\,\,\hbox{and}\,\,\widehat{L}_0\,\,\hbox{is
finite dimensional with dim}\widehat{L}_0\leq m.
\end{equation}
\end{lemma}
\begin{proof}
Define $H_i(z)=\big<p,f_i(x)\big>$ and the corresponding
characteristic field
\begin{equation}\label{1.34}
Z_i(z)=\left(
         \begin{array}{c}
           f_i(x) \\
           A_i(x)p \\
         \end{array}
       \right),\,\,A_i(x)=-[\pl_x f_i(x)]^*,\,1\leq i\leq m.
\end{equation}
Compute a Lie bracket
\begin{equation}\label{1.35}
[Z_i,Z_j](z)=\left(
               \begin{array}{c}
                 [f_i,f_j](x) \\
                 P_{ij}(z) \\
               \end{array}
             \right),\,\,i,j\in\{1,\ldots,m\},\,\,x\in\b,
\end{equation}
where $[f_i,f_j]$ is the Lie bracket using
$f_i,f_j\in\mathcal{C}^\infty(\b,\mb^n)$. Here $P_{ij}(z)$ is
computed as follows
\begin{equation}\label{1.36}
\begin{array}{ll}
    P_{ij}(z)&= [\pl_x(A_j(x)p)]f_i(x)+A_j(x)A_i(x)p-[\pl_x A_i(x)p]f_j(x)\\
    &=-A_i(x)A_j(x)p=\pl_x\big[ <A_j(x)p,f_i(x)>-<A_i(x)p,f_j(x)>\big]\\
    &=\pl_x\big<p,[f_i,f_j](x)\big>
  \end{array}
\end{equation}
where
$P_{ij}(z)\mathop{=}\limits^{\hbox{def}}\big[\pl_z(A_j(x)p)\big]Z_i(z)-\big[\pl_z(A_i(x)p)\big]Z_j(z),\,\,x\in
B(x_0,2\rho)$, is used. It shows (see \eqref{1.24}) that the Lie
bracket $[Z_i,Z_j](z)$ in \eqref{1.35} can be written as a
characteristic vector field associated to
$H_{ij}(z)\mathop{=}\limits^{\hbox{def}}\big<p,[f_i,f_j](x)\big>,\,x\in\b\s\mb^n$.
As a consequence, assuming that $\{f_1,\ldots,f_m\}$ satisfies
\eqref{1.31} and \eqref{1.32}, we get that
$\{Z_1,\ldots,Z_m\}\s\widehat{K}_0$ is a system of generators for
$\widehat{K}_0$. The proof is complete.
\end{proof}
\begin{lemma}\label{l.2}
Assume that PDE \eqref{1.1} is determined by a second order
continuously differentiable $H_0(x,p):\b\rt\mb$, and the Lie algebra
$\widehat{L}_0=\widehat{K}_0$ defined in Remark \ref{r.3} is of the
finite type over $\mb$. Then there exists a parameterized stationary
solution of PDE \eqref{1.1} given by
\begin{equation}\label{1.37}
\widetilde{z}(\l,z_0)=G_1(t_1)\circ\ldots\circ
G_m(t_m)[z_0],\,\,\l=(t_1,\ldots,t_m)\in\mathop{\prod}\limits_{i=1}^{m}[-a_i,a_i]=\Lambda,
\end{equation}
where $\{G_i(\sigma)[y]:\,\,\sigma\in[-a_1,a_i],\,\,y\in
B(z_0,\rho)\}$ is the local flow generated by the vector field
$Z_i\in \mathcal{C}^\infty(\b,\mb^{2n})$ and $\{Z_1,\ldots,Z_m\}\s
\mathcal{C}^\infty (\b;\mb^{2n})$ is a system of generators for
$\widehat{L}_0$.
\end{lemma}
\begin{proof}
Assuming that $\widehat{L}_0$ is of the finite type over $\mb$, we
notice that fixing a system of generators
$\{Z_1(z),\ldots,Z_m(z):z\in\b\s\mb^{2n}\}\s\widehat{K}_0$ for
$\widehat{L}_0$ we may and do construct a finite dimensional Lie
algebra $L(Z_1,\ldots,Z_m)\s\widehat{L}_0$ for which
$\{Z_1,\ldots,Z_m\}\s \widehat{K}_0$ is a system of generators over
$\mb$. It allows (see \cite{03}) to define a corresponding gradient
system in the finite dimensional Lie algebra $L(Z_1,\ldots,Z_m)$
associated with the following composition of local flows
\begin{equation}\label{1.38}
\widehat{z}(\l,z_0)=G_1(t_1)\circ\ldots\circ
G_m(t_m)[z_0],\,\,\l=(t_1,\ldots,t_m)\in\mathop{\prod}\limits_{i=1}^{m}[-a_i,a_i]=\Lambda.
\end{equation}
Here $\{G_i(\sigma)[y]:\,\,\sigma\in[-a_1,a_i],\,\,y\in
B(z_0,\rho)\}$ is the local flow generated by the vector field
$Z_i\in \widehat{K}_0$ and $\{Z_1,\ldots,Z_m\}$ is the system of
generators. In addition, there exists analytic vector fields
$\{q_1(\l),\ldots,q_m(\l):\,\l\in
B(0,a)\s\Lambda\}\s\mathcal{C}^{\omega}(B(0,a);\mb^m)$ such that
each vector field $Z_i\big(\widehat{z}(\l,z_0)\big),\,\l\in B(0,a)$,
can be recovered by taking the Lie derivative
\begin{equation}\label{1.39}
\pl_\l\widehat{Z}(\l;z_0)q_i(\l)=Z_i\big(\widehat{z}(\l.z_0)\big),\,\,\l\in
B(0,a)\s\Lambda,\,\,\,i\in\{1,\ldots,m\}
\end{equation}
and
\begin{equation}\label{1.40}
\{q_1(\l),\ldots,q_m(\l)\}\s\mb^m\,\,\hbox{are linearly independent
for any}\,\,\l\in B(0,a)\s\Lambda.
\end{equation}
The manifold defined in \eqref{1.38} stands for the parameterized
stationary solution of PDE \eqref{1.1} and by a direct computation,
we get
\begin{equation}\label{1.41}
\big<\pl_\l H_0\big(\widehat{z}(\l,z_0)\big),q_i(\l)\big>=\big<\pl_z
H_0\big(\widehat{z}(\l,z_0)\big),Z_i\big(\widehat{z}(\l,z_0)\big)\big>=0
\end{equation}
for any $\l\in B(0,a)\s\Lambda$ and $i\in\{1,\ldots,m\}$. Using
\eqref{1.40}, we notice that \eqref{1.41} lead us to
\begin{equation}\label{1.42}
\pl_\l H_0\big(\widehat{z}(\l,z_0)\big)=0,\,\,\forall\,\l\in
B(0,a)\s\Lambda
\end{equation}
and the proof is complete.
\end{proof}
\section{Main results}
With the same notations as in section 2 and considering that Lie
algebra $\widehat{L}_0=\widehat{K}_0$, defined in Remark \ref{r.3},
is of the finite type over $\mathcal{C}^\infty (\b,\mb^{2n})$ (see
definition \ref{d.2}), we get
\begin{theorem}\label{th.1}
Assume that PDE \eqref{1.1} is determined by
$H_0\in\mathcal{C}^2(\b\s\mb^{2n})$ and the Lie algebra
$\widehat{L}_0=\widehat{K}_0$ is of the finite type over
$\mathcal{C}^\infty(\b\s\mb^{2n})$. Then there exists a
parameterized stationary solution of PDE \eqref{1.1} given by
\begin{equation}\label{1.43}
\widehat{z}(\l,z_0)=G_1(t_1)\circ\ldots\circ
G_m(t_m)[z_0],\,\,\l=(t_1,\ldots,t_m)\in
B(0,a)\s\mathop{\prod}\limits_{i=1}^{m}[-a_i,a_i]=\Lambda.
\end{equation}
where $\{G_i(\sigma)[y]:\,\,\sigma\in[-a_1,a_i],\,\,y\in
B(z_0,\rho)\}$ is the local flow generated by the vector field
$Z_i\in \mathcal{C}^\infty(\b;\mb^{2n})$ and
$\{Z_1,\ldots,Z_m\}\s\mathcal{C}^\infty(\b;\mb^{2n})$ is the system
of generators for $\widehat{L}_0$.
\end{theorem}
\begin{proof}
By hypothesis, let
$\{Z_1,\ldots,Z_m\}\s\mathcal{C}^\infty(\b;\mb^{2n})$ be a system of
generators for $\widehat{L}_0$ which is assumes of the finite type
over $\mathcal{C}^\infty(\b;\mb^{2n})$. Define an orbit of
$\widehat{L}_0$ starting from $z_0$.
\begin{equation}\label{1.44}
\widetilde{z}(\l,z_0)=G_1(t_1)\circ\ldots\circ
G_m(t_m)[z_0],\,\,\l=(t_1,\ldots,t_m)\in\mathop{\prod}\limits_{i=1}^{m}[-a_i,a_i]=\Lambda.
\end{equation}
where $\{G_i(\sigma)[y]:\,\,\sigma\in[-a_1,a_i],\,\,y\in
B(z_0,\rho)\}$ is the local flow generated by the vector field
$Z_i\in\{Z_1,\ldots,Z_m\},\,\,i\in\{1,\ldots,m\}$. Let
$L(Z_1,\ldots,Z_m)\s\mathcal{C}^\infty(\b;\mb^{2n})$ be the Lie
algebra generated by $\{Z_1,\ldots,Z_m\}$ and notice that
$L=L(Z_1,\ldots,Z_m)$ is finitely generated over orbits starting
from $z_0$, which is abbreviated as $(f.g.o;z_0)$ in \cite{03}. On
the other hand, using the orbit starting from $z_0$ defined in
\eqref{1.44}, we associate a gradient system in $L(Z_1,\ldots,Z_m)$.
\begin{equation}\label{1.45}
\begin{array}{ll}
    \pl_1\widehat{z}(\l,z_0)&= Z_1\big(\widehat{z}(\l,z_0)\big),\,\pl_2 \widehat{z}(\l,z_0)=X_2\big(t_1;\widehat{z}(\l,z_0)\big),\ldots,\\
    &=\pl_1\widehat{z}(\l,z_0)=X_m\big(t_1,\ldots,t_{m-1};\widehat{z}(\l,z_0)\big)
    ,\,\l=(t_1,\ldots,t_m)\in\\\mathop{\prod}\limits_{i=1}^{m}[-a_i,a_i]=\Lambda
  \end{array}
\end{equation}
where
$\pl_i\widehat{z}(\l,z_0)\mathop{=}\limits^{\hbox{def}}\pl_{t_i}\widehat{z}(\l,z_0),\,\,i\in\{1,\ldots,m\}$.

Using the algebraic representation of a gradient system determined
by a system of generators $\{Z_1,\ldots,Z_m\}\subset L$ in a
$(f.g.o;z_0)$ Lie algebra $L$ (see \cite{03}), we get
\begin{equation}\label{1.46}
\pl_\l\widehat{z}(\l,z_0)=\{Z_1,\ldots,Z_m\}(\widehat{z}(\l,z_0))A(\l),\,\l\in\Lambda,\,\,A(0)=I_m,
\end{equation}
where the $(m\times m)$ matrix $A(\l)$ is nonsingular for any $\l\in
B(0,a)\s\Lambda$ with some $a>0$. It lead us to get
$\big\{\{q_1(\l),\ldots,q_m(\l)\}\s \mb^m :\,\l\in B(0,a)\big\}$
such that
\begin{equation}\label{1.47}
\{q_1,\ldots,q_m\}\s\mathcal{C}^\infty(B(0,a);\mb^m)\,\,\,\hbox{are
linearly independent}\,\,\forall\,\l\in B(0,a),
\end{equation}

\begin{equation}\label{1.48}
\pl_\l\widehat{z}(\l,z_0)q_i(\l)=Z_i\big(B(0,a);\mb^m\big),\,\,\l\in
B(0,a)\s \Lambda,\,i\in\{1,\ldots,m\}.
\end{equation}
The equation \eqref{1.48} allows us to get the conclusion
\begin{equation}\label{1.49}
\pl_\l\widehat{z}(\l,z_0)q_i(\l)=0,\,\,\l\in B(0,a)\s
\Lambda,\,i\in\{1,\ldots,m\}.
\end{equation}
and using \eqref{1.47}, we obtain $\pl_\l
H_0(\widehat{z}(\l,z_0))=0,\,\l\in B(0,a)\s\Lambda$ and the proof is
complete.
\end{proof}

With the same notations as in section 2 and considering that the Lie
algebra $L_0$ is of the finite type over
$\mathcal{C}^\infty(\b,\mbn)$ with respect to $K_0\s L_0$ (see
definition \ref{def.1}), we get
\begin{theorem}\label{th.2}
Assume that PDE \eqref{1.1} is determined by
$H_0\in\mathcal{C}^\infty(\b\s\mbn)$ and $L_0\supset K_0$ is of the
finite type over $\mathcal{C}^\infty(\b\s\mbn)$ with respect to
$K_0$. Let
$\{Z_1,\ldots,Z_m\}\s\mathcal{C}^\infty(\b\s\mbn,\mbn)\cap K_0$ be a
system of generators for $L_0$. Then there exists a parameterized
stationary solution for PDE \eqref{1.1} given by
\begin{equation}\label{1.50}
\widehat{z}(\l,z_0)=G_1(t_1)\circ\ldots\circ
G_m(t_m)[z_0],\,\,\l=(t_1,\ldots,t_m)\in
B(0,a)\s\mathop{\prod}\limits_{i=1}^{m}[-a_i,a_i]=\Lambda.
\end{equation}
where $\{G_i(\sigma)[y]:\,\,\sigma\in[-a_1,a_i],\,\,y\in
B(z_0,\rho)\}$ is the local flow generated by the vector field
$Z_i,\,1\leq i\leq m$.
\end{theorem}
\begin{proof}
By hypothesis, any Lie bracket $[Z_i,Z_j](z)$ is a linear
combination of $\{Z_1,\ldots,Z_m\}(z)$, $z\in B(z_0,2\rho)$, using
smooth functions from $\mathcal{C}^\infty(\b\s\mbn)$. It shows that
the Lie algebra
$L=L(Z_1,\ldots,L_m)\s\mathcal{C}^\infty(\b\s\mbn;\mbn)$ generated
by the fixed $\{Z_1,\ldots,Z_m\}$ is finitely generated over orbits
starting from $z_0$ (see $(f.g.o;z_0)$ Lie algebra in \cite{03}). In
addition, $\{Z_1,\ldots,Z_m\}$ is a system of generators for $L$.
Using the orbit defined in \eqref{1.50} we proceed as in proof of
Theorem \ref{th.1} (see \eqref{1.45}-\eqref{1.49}) and get the
conclusion. The proof is complete.
\end{proof}
\textbf{Acknowledgement}\\
We are much obliged to Professor Dr.\ C. V\^{a}rsan  for this
research article.

\vskip 1cm

\hspace{3.1in}\textit{Department of Mathematics},

\hspace{2.8in}\textit{Government College University,}

\hspace{3.3in} \textit{Faisalabad, Pakistan.}

\hspace{3.2in} \email {(\textit{saimashaa@gmail.com})}\\

\hspace{3.8in} and\\

\hspace{3.1in}\textit{Department of Mathematics},

\hspace{1.7in} \textit{COMSATS Institute of Information Technology,}

\hspace{3.3in}\textit{Islamabad, Pakistan.}

 \hspace{3.1in}\email{\textit{(mrsaeedakram@gmail.com)}}

\end{document}